\documentclass[11pt]{amsart}
\usepackage{amssymb,amsmath,amsthm,latexsym}
\usepackage{amssymb,graphicx}
\usepackage{color}
\newtheorem{theorem}{Theorem}
\newtheorem{theorem*}{Theorem 4}

\newtheorem{conjecture}[theorem]{Conjecture}
\newtheorem{corollary}[theorem]{Corollary}

\newtheorem{definition}[theorem]{Definition}

\newtheorem{proposition}[theorem]{Proposition}

\newtheorem*{mainthm}{Theorem \ref{betti}}
\newtheorem*{mainthmtwo}{Theorem \ref{tsubgroup}}
\newtheorem*{amthm}{Theorem \ref{almostmalnormal}}

\def\field{\mathbb{F}_2}
\def\coind{\textrm{Coind}}
\def\ind{\textrm{Ind}}
\def\hom{\textrm{Hom}}

\title{Relative Ends, $\ell^2$-Invariants and Property (T)}
\author{Aditi Kar and Graham A. Niblo}
\address{School of Mathematics, University of Southampton, Highfield, Southampton, SO17 1SH, England}
\email{A.Kar@soton.ac.uk, G.A.Niblo@soton.ac.uk}
\thanks{This research was partially supported by EPSRC grant  EP/F031947/1.}

\begin{document}
\begin{abstract}
We establish a splitting theorem for one-ended groups $H \leq G$ such that $\tilde{e}(G,H) \geq 2$ and the almost malnormal closure of $H$ is a proper subgroup of $G$. This yields splitting theorems for groups $G$ with non-trivial first $\ell^2$-Betti number $\beta^2_1(G)$. We verify the Kropholler Conjecture for pairs $H\leq G$ satisfying $\beta^2_1(G) > \beta^2_1(H)$. We also prove that every $n$-dimensional Poincar\'e duality ($PD^n$) group containing a $PD^{n-1}$ group $H$ with property (T) splits over a subgroup commensurable with $H$.
\end{abstract}

\maketitle

In this article we explore the relationship between the theory of relative ends, groups with non-trivial first $\ell^2$-cohomology and the presence of subgroups with property (T). The desired conclusion is to obtain splittings of groups, i.e., nontrivial decompositions of groups into amalgams or HNN extensions. We use two different notions of `relative ends' for groups $H \leq G$, the geometric one which is usually written $e(G,H)$ and its algebraic counterpart $\tilde{e}(G,H)$.

The classical theory of the ends of a group originated in the work of Freudenthal and Hopf (See \cite{ends}, \cite{hopf}). From the point of view of a geometric group theorist the number of ends of a finitely generated group $G$, written $e(G)$,  is the number of Freudenthal-Hopf ends of a connected locally finite Cayley graph for $G$, regarded as a 1-dimensional simplicial complex. While \emph{a priori} the number could depend on the generating set chosen, it is in fact independent provided the chosen generating set is finite, i.e., it is a quasi-isometry invariant of the group. There is an alternative definition of $e(G)$ which is more obviously independent of choice of generating sets, and which extends to a definition of the number of ends for an arbitrary discrete group. 

\begin{definition}
Let $G$ be a discrete group, $\mathcal P(G)$ denote the power set of $G$, and $\mathcal F(G)$ denote the set of finite subsets of $G$. Then $\mathcal{F}(G), \mathcal {P}(G)$ and the quotient $ \mathcal{F}(G) \backslash \mathcal{P}(G)$ are all $\field G$-modules, where $\field$ denotes the field of $2$ elements. We denote by $e(G)$ the dimension of the $G$ invariant subspace $ \left(\mathcal{F}(G)) \backslash \mathcal{P}(G) \right)^G$.
\end{definition}

Hopf showed in \cite{hopf} that the number of ends of a finitely generated group must be 0, 1, 2 or $\infty$. Moreover, groups with 0 and 2 ends are easily classified: $e(G)=0$ if and only if $G$ is finite and $e(G)=2$ if and only if $G$ is virtually $\mathbb{Z}$. Stallings' celebrated theorem from \cite{infends} classifies finitely generated groups for which $e(G)\geq 2$. We state it here in its most general form as proved by Dicks and Dunwoody  using the Almost Stability Theorem. 
 
\begin{theorem} \label{stallings}(Theorem IV.6.10 of \cite{dicks}) Let $G$ be a group. The following are equivalent:
\begin{enumerate} 
\item $e(G) > 1$ 
\item $H^1(G,M) \neq 0$, for any free $G$ module $M$, 
\item There exists a $G$-tree with finite edge stabilizers such that no vertex is stabilized by $G$. 
\item One of the following holds:
\begin{itemize}
\item $G=B*_CD$ where $B \neq C \neq D$ and $C$ is finite, 
\item $G=B*_C$, where $C$ is finite, 
\item $G$ is countably infinite and locally finite. 
\end{itemize} 
\item the group $G$ has 2 or infinitely many ends. 
\end{enumerate}
\end{theorem} 
 
The quest for a generalisation of this result covering splittings over arbitrary subgroups has played a central role in low dimensional topology and geometric group theory. The classical and algebraic annulus and torus theorems are key examples (See \cite{Scott2} and references therein). While working on this problem, Scott introduced in \cite{scottends} an invariant $e(G,H)$ for a subgroup $H$ of a group $G$, which, in the case when $G$ is finitely generated, can be identified with the number of Freudenthal-Hopf ends of the quotient of a locally finite Cayley graph for $G$ by the action of $H$. As with the classical end invariant, $e(G,H)$ does not depend on the choice of Cayley graph, and indeed the definition may be extended to the class of all discrete groups. We postpone the definition to section \ref{relativeends}.
  
Scott showed in \cite{scottends} that if $G$ splits as a non-trivial amalgamated free product $G=A*_CB$ or as an HNN extension $G=A*_C$ then $e(G,H)\geq 2$. Noting that $e(G,\{1\})=e(G)=e(G,C)$ for any finite subgroup $C<G$, Scott reformulated Stallings' theorem as the statement that $G$ splits over a finite subgroup if and only if $e(G,C)\geq 2$ for some finite subgroup $C<G$. He asked for which subgroups $H<G$ the analogous statement is true, remarking that it is certainly not true in general. For example the triangle group $G=\langle a,b,c\mid a^2=b^2=c^2=(ab)^2=(bc)^3=(ca)^5\rangle$ has an infinite cyclic subgroup $H=\langle ab^{-1}\rangle$ with $e(G,H)=2$, but the group $G$ does not split as an HNN extension, nor as a non-trivial amalgamated free product, over \emph{any} subgroup. Scott's resolution to this was the observation that while $G$ does not split, it has a finite index subgroup $G'$ which splits as an HNN extension over $H$. 

A more complete answer was given by the algebraic annulus theorem which asserts that if $G$ is a one ended finitely generated group containing a two-ended subgroup $H$ with $e(G,H)\geq 2$ then $G$ is virtually $\mathbb{Z}^2$ or $G$ contains a two ended subgroup $K$ over which it splits, or $G$ has a finite normal subgroup $N$ whose factor group is a surface group. Here, we see two ways in which the obstruction to splitting over a subgroup can be overcome: one is to replace the group $G$ by a finite index subgroup, the other is to adjust the subgroup $H$. Both strategies play an important role in low dimensional topology. The latter is crucial in the statement and proof of the classical torus theorem (the fore-runner of the algebraic annulus and torus theorems) while the former is related to the virtual Haken and virtually positive first Betti number conjectures.
 
Scott's proof that the triangle group contains a finite index subgroup which splits over the infinite cyclic subgroup $H$ relied on the observation that the subgroup $H$ is an intersection of finite index subgroups. Scott generalised this in \cite{scottsurface} to show that if $G$ is a finitely generated group, and $H<G$ is a finitely generated subgroup which is an intersection of finite index subgroups and such that $e(G,H)\geq 2$ then $G$ has a finite index subgroup which splits over $H$. In particular, if $G$ is a LERF group (i.e., a group in which every finitely generated subgroup is an intersection of finite index subgroups of $G$), then every finitely generated subgroup $H$ with $e(G,H)\geq 2$ is the edge group of a splitting for some finite index subgroup of $G$. Essentially the idea is that the obstruction to splitting $G$ over $H$ (sometimes referred to as the singularity obstruction) is carried by finitely many double cosets of $H$ in $G$ and that by passing to a suitable finite index subgroup one removes all these elements.

In \cite{N} the singularity obstruction $\mathcal S=\text{Sing}(G,H)$ was studied in more depth and it was shown that if $\mathcal S\cup H$ is contained in a proper subgroup $G'$ of $G$ then $G$ will split over a subgroup  of the group $\langle \mathcal S\cup H\rangle$, while if  $\mathcal S$ is contained in the commensurator of $H$ in $G$ then $G$ will split over a subgroup commensurable with $H$. Scott's technique of passing to finite index subgroups was also strengthened to show that if the singularity obstruction is supported on $n$ double cosets of $H$ in $G$ and $H$ is contained in a strictly decreasing chain of finite index subgroups of $G$ of length at least $n$ then $G$ has a finite index subgroup which splits.

While this last result has the advantage that it no longer requires $H$ to be an intersection of finite index subgroups, the length of the chain required to ensure that $G$ virtually splits depends crucially on the size of the splitting obstruction and therefore, on the embedding of $H$ in $G$. In an effort to circumvent this difficulty we offer the following result (Corollary to Theorem \ref{betti}) which replaces the size of the singularity obstruction in the statement by a number which depends on the $\ell^2$ Betti numbers of $H$ and $G$ instead. This has the advantage that it is intrinsic to the groups $H$ and $G$ and does not depend on the embedding of $H$ in $G$, but comes with the disadvantage of applying only when $G$ has positive $\ell^2$ Betti number, $\beta_1^{(2)}(G)$. See \cite{malnormal} for examples. 

\begin{corollary}\label{positiveGbetti} Let $H \leq G$ be discrete and countable one-ended groups such that $\beta_1^{(2)}(G) >0$.  If $\tilde{e}(G,H) \geq 2$ and $H$ is contained in a finite index subgroup $G'<G$ with $[G:G'] > \beta_1^{(2)}(H)/\beta_1^{(2)}(G)$,  then $G'$ splits over a subgroup of the almost malnormal closure of  $H$. (See Definition \ref{am}.)\end{corollary}
 
The end invariant $\tilde{e}(G,H)$ mentioned above is a generalisation of Scott's end invariant and was introduced by Kropholler and Roller, \cite{KR}, in their study of the algebraic torus theorem for Poincar\'e duality groups. We will state the definition of $\tilde{e}(G,H)$ in section \ref{relativeends}, but note here that in particular if $e(G,H)\geq 2$ then $\tilde{e}(G,H)\geq 2$ as required.

For an introduction to the theory of $\ell^2$ cohomology, we refer the reader to \cite{luck2}. Corollary \ref{positiveGbetti} follows directly from Theorem \ref{betti} below. Note that groups with non-trivial first $\ell^2$ betti number are either one-ended or have infinitely many ends. In the latter case, Theorem \ref{stallings} says that the group splits over a finite subgroup or is locally finite. 

\begin{theorem} \label{betti} Let $H \leq G$ be discrete and countable one-ended groups such that $\beta_1^{(2)}(G)$ $> \beta_1^{(2)}(H)$. If $\tilde{e}(G,H) \geq 2$ then $G$ splits over a subgroup of the almost malnormal closure of  $H$. \end{theorem}

\noindent \textbf{Coxeter Groups} We now provide explicit examples in which the hypotheses of Theorem \ref{betti} are satisfied using the theory of Coxeter groups. Niblo and Reeves have shown in \cite{catcoxeter} that every finitely generated Coxeter group $W=W(S)$ acts properly discontinuously on a locally finite, finite dimensional CAT(0) cube complex $X_W$. Sageev's work on ends of group pairs then implies that $e(W, H)\geq 2$ for each wall stabiliser $H<W$ and it is easy to deduce that $H$ is the centralizer of a reflection. If $W$ is a Coxeter group with $\beta^2_1(W)  \neq 0$, then one can extract additional information about the structure of $W$ using Theorem \ref{betti} and Corollary \ref{positiveGbetti}. 

To start with, let $W$ be the Coxeter group generated by the reflections $s_1, \ldots, s_8$ such that $s_1$ commutes with each of $s_4$, $s_5$ and $s_6$ while the pairwise product of $s_1$ with each of $s_2$, $s_3$, $s_7$ and $s_8$ is of infinite order. The pairwise products of the generators $s_4$, $s_5$ and $s_6$ are of order 3. The remaining pairwise products are finite but greater than $50$. Then $W$ is a one-ended Coxeter group whose first $\ell^2$ betti number is non-zero, as can be seen from applying Theorem 3.2 of \cite{malnormal}. 

Nuida describes the centralizers of reflections in his paper \cite{nuida} and from his work, one deduces that the centralizer $C$ of the reflection $s_1$ is precisely $T(3,3,3) \times \langle s_1 \rangle$. Here, $T(3,3,3)$ is the triangle group obtained from the parabolic subgroup generated by  $s_4$, $s_5$ and $s_6$. As explained earlier, $e(W,C) \geq 2$. Moreover $C$ contains $\mathbb{Z}^2$ as a finite index subgroup and therefore $\beta^2
_1(C)=0 $. 

Using the same strategy one can build a whole family of examples using the hyperbolic triangle groups $T(p,q,r)$, where $p$, $q$ and $r$ are positive integers satisfying $\frac{1}{p} + \frac{1}{q} + \frac{1}{r} < 1$. This time, take $W_n$ to be a Coxeter group generated by $n$ reflections, $s_1$, $\dots, s_n$. As in the earlier example, the reflection $s_1$ commutes with precisely 3 other reflections $s_2$, $s_3$ and $s_4$ while the product of $s_1$ with each of $s_5, \ldots, s_n$ has infinite order. For sake of simplicity, we set the order of all pairwise products not already specified to be $n^2$. As before the centralizer $C(s_1)$ is precisely $T(p,q,r) \times \langle s_1 \rangle$ and $e(W_n,C(s_1)) \geq 2$. Using Theorem 3.2 of \cite{malnormal} again, we have 
\[\beta^2_1(W) \geq \frac{n}{2} -1 - \left(\frac{3}{2}+ \frac{1}{p} + \frac{1}{q} +\frac{1}{r} + \frac{1}{n^2}\left(\frac{n(n-1)}{2}-\left(n-1 +3\right)\right)\right)\]

\noindent Now $\beta^2_1(C(s_1))$ is one-half of $\beta^2_1(T(p,q,r))$. Let $\chi (.)$ denote the orbifold Euler characteristic of a group. One computes that \[\chi(T(p,q,r))= \frac{1}{2}\left( \frac{1}{p} + \frac{1}{q} + \frac{1}{r} -1\right)\ \ \ \ \] 
\noindent Moreover, $\beta ^2_1 (T(p,q,r))$ $=-\chi(T(p,q,r))$. This is a consequence of Atiyah's formula relating the $\ell^2$-Euler characteristic to the orbifold Euler characteristic. But for Fuchsian groups and in particular triangle groups, the argument may be simplified. Every triangle group contains a surface subgroup of finite index. Suppose $T(p,q,r)$ contains a surface subgroup $H \cong \pi_1(S_g)$ (here, $g$ is the genus) of index $k$. From first principles, $\beta^2_1(H)= -\chi(S_g)$. Now, both $\beta^2_1(.)$ and $\chi(.)$ are multiplicative on indices hence
\[\beta^2_1(T(p,q,r))=k \beta^2_1(H) = k (-\chi(S_g))=- \chi(T(p,q,r))\]

\noindent Given $p$, $q$ and $r$, for $\beta^2_1(W) > \beta^2_1(C(s_1))$ to hold, we need \[ \frac{1}{2}\left( n - 6 + \frac{3}{n} + \frac{4}{n^2} \right) -\left( \frac{1}{p} + \frac{1}{q} + \frac{1}{r} \right)> -\frac{1}{2}\chi(T(p,q,r)) \]
\noindent In particular if $n-6 > 3\chi(T(p,q,r)) + 2$ then $\beta^2_1(W_n) > \beta^2_1(C(s_1))$.

One may specialise to the well-known $(2,3,7)$ triangle group, which contains the fundamental group of the Klein's quartic (a surface of genus 3) as a subgroup of index 336. Since $\beta^2_1(T(2,3,7))= \frac{1}{84}$, one can choose $n$ to be 8 and get a splitting of $W_8$ over $T(2,3,7) \times \mathbb{Z}/2\mathbb{Z}$. This splitting may also be obtained from visual decompositions of Coxeter groups into amalgams.  

\medskip

It is worth noting here that the proof of Theorem \ref{betti} applies in more generality. 

\begin{definition}\label{am} We will say that a subgroup $H$ of a group $G$ is \textit{almost malnormal} if for every $g \notin H$, the intersection $H \cap H^g$ is finite. 

The almost malnormal closure of a subgroup $H<G$ is the  intersection of the almost malnormal subgroups of $G$ containing $H$.
 \end{definition}
\medskip

We have the following generalisation of \cite[Theorem 4.9]{torus}. 

\begin{theorem} \label{almostmalnormal}
Let $H \leq G$ be one-ended groups such that $\tilde{e}(G,H) \geq 2$. If the almost malnormal closure $K$ of $H$  is not equal to $G$ then $G$ splits over a subgroup of $K$. 
\end{theorem}

The \emph{Kropholler conjecture} is a long standing conjecture of Kropholler and Roller from \cite{KR}. To read more about the current status of the conjecture, see \cite{Nibloshort}. We show that our techniques give further evidence towards the conjecture by verifying it for pairs of groups $H \leq G$ satisfying $\beta^2_1(G) > \beta ^2_1(H)$. This is the content of Proposition \ref{conj2}. 

\medskip

The main protagonists of our next theorem are Poincar\'e duality groups. An introduction to the notion of Poincar\'e duality may be found in \cite{torus}. Fundamental groups of closed aspherical manifolds are Poincar\'e duality groups. Whether the converse is true for finitely presented groups is the subject of a well known conjecture. One can show that the only one-dimensional Poincar\'e duality group is $\mathbb{Z}$. That all Poincar\'e duality of dimension 2 are surface groups is a deep theorem established by Bieri, Eckmann, Muller and Linnell. For each $n \geq 4$, Bestvina-Brady groups provide examples of Poincare duality groups which are not finitely presented and hence are not fundamental groups of closed aspherical manifolds. 

We provide the following splitting theorem for Poincar\'e duality groups which may be viewed as an analogue of the torus theorem and which plays a central role in the topological superrigidity theorem established in \cite{karniblotop}. 

\begin{theorem}  \label{tsubgroup} Let $G$ be a Poincar\'e duality group of dimension $n$. Suppose that $H$ is an $(n-1)$-dimensional Poincar\'e duality subgroup of $G$ and that $H$ has property (T). Then $G$ splits over a subgroup commensurable with $H$. \end{theorem} 

For example suppose that $M$ is a closed aspherical manifold of dimension $4n+1$, $n \geq 2$ and that $N$ is a quarternionic hyperbolic closed manifold of dimension ${4n}$ which admits a $\pi_1$-injective map into $M$.  Since $\pi_1(N)$ has property $(T)$ the theorem shows that $\pi_1(M)$ is a non-trivial amalgam or HNN extension over a subgroup commensurable with $\pi_1(N)$. Note that the presence of a codimension one property (T) subgroup in Theorem \ref{tsubgroup} becomes an obstruction to the ambient group having property (T). 

The paper is organised as follows. In section \ref{relativeends} we expand on the formal definition of the two end invariants $e(G,H)$ and $\tilde{e}(G,H)$ alluded to above. In section \ref{bettisection} we give the proof of Theorems \ref{betti} and \ref{almostmalnormal} and discuss the Kropholler conjecture. In section \ref{PropertyT} we deal with Poincar\'e duality and establish Theorem \ref{tsubgroup}.

\medskip

\noindent \textbf{Acknowledgements } We are grateful to Peter Kropholler, Indira Chatterji and Ashot Minasyan for their comments and suggestions.

\section{Relative Ends}\label{relativeends}

Throughout the paper we will denote the field of order two by $\field$. Now let $G$ be a group and $H$ be a subgroup of $G$. Given an $H$ module $M$ one may form a $G$ module using the functors $\textrm{Hom}_H(\field[G],\ \ )$ and $\field[G] \otimes_H$. More precisely, choosing a set $S$ of right coset representatives for $H \leq G$ we have
\[\coind ^G _H M := \hom_H(\field[G], M) \cong \prod_{g \in S} Mg \]
\[\ind ^G _H M := \field[G] \otimes_H M \cong \bigoplus_{g \in S} Mg\]

Let $\mathcal{P}G$ denote the collection of all subsets of $G$. Then, $\mathcal{P}G$ is an $\field$-vector space with respect to the operation of symmetric difference. One checks that $\mathcal{P}G$ is also a $G$ module. Moreover, $\mathcal{P}G \cong$ $\coind ^G _1 \field$. On the other hand \[\mathcal{F}_H (G)= \{A \subseteq G: A \subseteq HF \textrm{\ for some finite set F} \}\] is the $\field G$-module $\ind ^G_H \mathcal{P}H$.
Similarly the power set $\mathcal{P}(H\backslash G)$ of $H \backslash G$ and the collection of finite subsets of $H\backslash G$, written $\mathcal{F}(H
\backslash G)$ are $\field[G]$ modules. In fact, $\mathcal{P}(H\backslash G)$ $\cong \coind ^G _H \field$ and $\mathcal{F}(H \backslash G)$ $\cong \ind ^G _H \field$. 

\begin{definition} The elements of $\mathcal{F}_H(G)$ are said to be \textit{$H$-finite} and the elements of $\left(\mathcal{F}_H(G) \backslash \mathcal{P}G \right) ^G$ are called \textit{$H$-almost invariant sets}. \end{definition}

\begin{definition} The algebraic end invariant is defined as \[\tilde{e}(G,H)=\dim_{\field} \left(\mathcal{F}_H(G) \backslash \mathcal{P}G \right) ^G \] while the geometric end invariant is defined as \[e(G,H)=\dim_{\field} \left(\mathcal{F}(H\backslash G)) \backslash \mathcal{P}(H\backslash G) \right)^G.\] \end{definition}

We collect together the properties of the end invariants defined above which we will later need. The interested reader may find more details in \cite{KR}.

\begin{proposition}\label{basics} Let $H \leq K \leq G$ be groups. Then the following hold.
\begin{enumerate}
\item $e(G,1)=e(G)=\tilde{e}(G,1)$. 
\item $e(G,H) = 0 = \tilde{e}(G,H)$ if and only if $H$ has finite index in $G$. 
\item If $H$ has infinite index then $\tilde{e}(G,H) $  $= 1 + \dim_{\field} H^1(G, \mathcal{F}_H(G))$. 
\item If $K$ has infinite index then $\tilde{e}(G,H)$ $\leq $ $\tilde{e}(G,K)$.
\item $e(G,H) = e(X)$, where $X$ is the coset graph of $G$ with respect to $H$. 
\item $e(G,H) \leq \tilde{e}(G,H)$. 
\end{enumerate}
\end{proposition}

Note that the algebraic end invariant for a group with infinitely many ends with respect to any of its infinite index subgroups is infinite. For instance, if $G$ is the non-abelian free group of rank 2 and $G'$ denotes its commutator subgroup, then $\tilde{e}(G,G') = \infty$ (whereas $e(G,G')=2$). Clearly, the algebraic end invariant gives useful information only about one-ended groups.

\section{Proof of Theorems \ref{betti} and \ref{almostmalnormal}} \label{bettisection}

\begin{mainthm} Let $H \leq G$ be discrete and countable one-ended groups such that $\beta_1^{(2)}(G)$ $> \beta_1^{(2)}(H)$. If $\tilde{e}(G,H) \geq 2$ then $G$ splits over a subgroup of the almost  malnormal closure of  $H$ in $G$. \end{mainthm}

Peterson and Thom showed in \cite{malnormal} that if $\beta_1^{(2)}(G)$ $> \beta_1^{(2)}(H)$ for a torsion free discrete countable group $G$ then there exists a proper malnormal subgroup $H'$ of $G$ that contains $H$. If one drops the hypothesis that $G$ is torsion free then the same argument shows that $H'$ is almost malnormal (see Definition \ref{am}). So Theorem \ref{betti} follows directly from Theorem \ref{almostmalnormal}.

\begin{amthm}
Let $H \leq G$ be one-ended groups such that $\tilde{e}(G,H) \geq 2$. If the almost malnormal closure $K$ of $H$  is not equal to $G$ then $G$ splits over a subgroup of $K$. 
\end{amthm}

\begin{proof} Let $H \leq G$ be one-ended groups such that $H$ is contained in a proper almost malnormal subgroup of $G$. 

Set $\Sigma= \{K < G : H \leq K \textrm{ and } K \textrm{ is almost malnormal in
}G\}$. Let $(K_j)_{j \in J}$ be elements of $\Sigma$ and suppose $g \notin \cap_{j \in J} K_j$. Then $g$ does not belong to $K_j$ for at least one $j \in J$. As $K_j$ is almost malnormal in $G$, $K_j \cap K_j^g$ is finite. Thus, $(\cap K_j) \cap (\cap K_j)^g $ is finite. We conclude that any intersection of elements of $\Sigma$ is almost malnormal and that $\Sigma$ has a minimal element, the almost malnormal closure of $H$ which we will denote $K$. We will now show that $\tilde{e}(G, K)\geq 2$ and $e(K)=1$.

As the subgroup $K$ is almost malnormal in $G$ and $G$ is infinite, $K$ has infinite index in $G$. As noted in Proposition \ref{basics}  the algebraic end invariant $\tilde{e}(G,.)$ is monotonic for infinite index subgroups, thus $\tilde{e}(G,K)$ $\geq \tilde{e}(G,H)$ and  $\tilde{e}(G,K)\geq 2$. 

The presence of a one ended subgroup $H$ in $K$ limits the possibilities for the value of $e(K)$. Firstly $K$ is infinite and so $e(K) \neq 0$. A group has two ends if and only if  it is virtually $\mathbb{Z}$. As $K$ has a subgroup which is not virtually $\mathbb{Z}$, $e(K) \neq 2$. Thus $K$ is either one ended or $K$ has infinitely many ends. The latter is not a possibility, as we will now show. 

Suppose that $K$ has infinitely many ends. Then by Theorem \ref{stallings}, $K$ acts on a tree $T$ with no global fixed point and so that edge stabilisers are finite. We may restrict the action to $H$, but since $e(H)=1$ this action does have a fixed point, and since $H$ is infinite it cannot fix an edge so it must have a fixed vertex. So $H<A=\text{Stab}_G(v)$ for some vertex $v$. We will show that $A$ is almost malnormal in $G$. As $K$ is minimal amongst the almost malnormal subgroups containing $H$, this will imply that $K <A$ and hence, $A=K$ which contradicts the fact that $K$ acts with no global fixed point on $T$. 

Suppose first that $k\in K\setminus A$. Then $kv\not=v$ so $A\cap A^k$ stabilises each edge on the non-trivial geodesic from $v$ to $kv$. It follows that $A\cap A^k$ is finite. This tells us what happens for elements of $G$ that lie in $K$. If $g \in G \backslash K$, then $K \cap K^g$ is finite and hence $A \cap A^g$ which is contained in $K \cap K^g$ is finite. Thus $A$ is almost malnormal in $G$. 

We now need to check that there exists a proper $K$ almost invariant subset $A$ in $G$ such that $AK=A$. We generalise Kropholler's methods in \cite {torus} to deal with the almost malnormal subgroups. The strategy will be to show that for our choice of $K$, $H^1(K, \mathcal{F}_K(G))=0$. Recall that $K$ is a one ended almost malnormal subgroup of $G$ such that $\tilde{e}(G,K) \geq 2$. 

Let $\Lambda$ be a set of representatives for the double cosets of $K$ in $G$. As a $K$ module, the induced module $\mathcal{F}_K(G)$ is given by \[\textrm{Res}^G _K \ind ^G _K \mathcal{P}K \cong \oplus_{g \in \Lambda} \ind ^K _{K \cap K^g}\textrm{Res}^{K^g} _ {K \cap K^g}\mathcal{P}Kg. \] 
The module $\textrm{Res}^{K^g} _ {K \cap K^g}\mathcal{P}Kg$ may be identified with $\textrm{Res}^{K} _{K^{g^{-1}} \cap K}\mathcal{P}K$. Now, let $g$ represent a non-trivial double coset of $K$ in $G$. Then, we have  \[\textrm{Res}^{K} _{K \cap K^g}\mathcal{P}K \cong \textrm{Res}^{K} _{K\cap K^g} \coind^K _1 \mathbb{F}_2 \cong \prod_{(K\cap K^g)\backslash K} \coind^{K \cap K^g} _1 \mathbb{F}_2\]
The subgroup $K \cap K^g$ is finite and so the module $\coind^{K \cap K^g} _1 \mathbb{F}_2$ is isomorphic to the module $\ind^{K \cap K^g} _1 \mathbb{F}_2$, which is precisely the group algebra $\mathbb{F}_2[K \cap K^g]$.

Let $R$ denote the algebra $\mathbb{F}_2[K \cap K^g]$. Since $R$ is finite, for any index set $I$, \[ R^I := \prod_I R  \cong R \otimes \mathbb{F}_2 ^I.\] To see this, observe that $R^I$ is the algebra of all $R$ valued maps on $I$. For any $f: I \rightarrow R$ and $ r \in R$, define $F(r)$ to be the set $\{i \in I : f(i) = r\}$. Then the assignment \[ f \mapsto \sum_{r \in R}  r\otimes F(r)\] is the required isomorphism. We deduce from this discussion that $R^I$ is a free module over the $\field$-group algebra and it follows that $\mathcal{P}Kg$ is a free $K \cap K^g$-module. A module induced from a free module is also free and so we find that $\mathcal{F}_K(G)$ is the direct sum of $\mathcal{P}K$ and a free module. By Shapiro's Lemma, $H^1(K,\mathcal{P}K)=0$ for all groups $K$. Moreover, by Theorem \ref{stallings}, the first cohomology group of the one ended group $K$ with respect to any free module is trivial. Thus, $H^1(K,\mathcal{F}_K(G))$ is zero. 

If $B$ is a proper $K$ almost invariant subset of $G$ and $H^1(K,\mathcal{F}_K(G))$  is zero, then the derivation $B \mapsto B+Bg$ restricts to a principal derivation on $K$. There exists then a $K$-finite subset $C$ such that $B+Bx = C+Cx$ for all $x \in K$. Choose $A$ to be $B+C$. 

Observe that for all $g \in G \backslash K$, $\tilde{e}(G, K \cap K^g) =1$. This is because $G$ is one ended and each of the intersections $K \cap K^g$ is finite. The theorem now follows directly from Theorem 5.3 of \cite{torus}.  \end{proof}

\subsection{A conjecture of Kropholler and Roller}
In the proof of Theorem \ref{betti} we used the non-vanishing of the kernel of the restriction map $Res^G _H$ from $H^1(G, \mathcal{F}_K G)$ to $H^1(H, \mathcal{F}_K G)$ to extract a bi-invariant proper $K$ almost invariant subset of $G$ and this in turn, helped to produce the splitting for the group. Kropholler and Roller conjectured the following: 

\begin{conjecture}\label{conj} (Kropholler and Roller, \cite{KR}) Let $H
\leq G$ be  finitely generated groups. If $G$ contains a proper $H$ almost invariant subset $A$ such that $HAH=A$, then $G$ splits over a subgroup related to $H$. \end{conjecture}

Here we provide further evidence in favour of the conjecture. 

\begin{proposition} \label{conj2} Conjecture \ref{conj} is true for all pairs $G$ and $H$ satisfying the hypotheses of the conjecture along with the condition $\beta_1^{(2)}(G)$ $> \beta_1^{(2)}(H)$. \end{proposition}

\begin{proof} The case when $H$ is finite follows from Stallings' celebrated Theorem on ends of groups. Assume that $H$ is infinite. Then, as before, $H$ is contained in a proper almost malnormal subgroup $K$ of $G$. 

Choose $A$ to be a proper $H$-almost invariant subset of $G$ such that $HAH=A$ and set $\mathcal{S}_A(G,H)$ to be the set of elements $g$ of the group such that all four intersections $A \cap gA$, $A \cap gA^*$, $A^* \cap gA$, and $A^* \cap gA^*$ are non-empty. This is the singularity obstruction defined in \cite{N} and discussed above.

By Kropholler's Lemma (4.17 of \cite{torus}), the condition that $A=AH$ ensures that $\mathcal{S}_A(G,H)$ is contained in the set $\mathcal{S}:=\{g \in G : \tilde{e}(G,H \cap H^g)\geq 2 \}$. Assume first that $\mathcal{S}$ is contained in $K$. Then, the singularity obstruction along with the subgroup $H$ generates a proper subgroup $\langle\mathcal{S} \cup H\rangle$ of $G$ and the main theorem of \cite{N} asserts that $G$ splits over a subgroup related to $\langle \mathcal{S}\cup H \rangle$ and hence to $H$. On the other hand, if  $\mathcal{S}$ is not contained in $K$, then for any $g \in \mathcal{S} \backslash K$, $\tilde{e}(G,H \cap H^g)\geq 2$ for the finite subgroup $H \cap H^g$. Once again, by Stallings theorem on ends of groups, $G$ splits over a subgroup commensurable with $H \cap H^g$. This verifies the conjecture for our choice of groups $G$ and $H$. \end{proof}

\section{Poincar\'e duality groups}\label{PropertyT}

\begin{mainthmtwo}
Let $G$ be a Poincar\'e duality group of dimension $n$. Suppose that $H$ is an $(n-1)$-dimensional Poincar\'e duality subgroup of $G$ and that $H$
has property (T). Then $G$ splits over a subgroup commensurable with $H$.
\end{mainthmtwo}

An $n$-dimensional Poincar\'e duality group is also called a $PD^n$ group. 

\begin{proof} Let $G$ and $H$ be as in the statement of the theorem. Then
a simple computation shows that the end invariant $\tilde{e}(G,H)$ is precisely 2. We include the computation here for sake of completeness. Recall that
$\tilde{e}(G,H)$ $=1 + \dim H^1(G, \mathcal{F}_H(G)$. Denote the dualizing module $H^n(G, \field G)$ by $D_G$. In our case, $D_G \cong \field$.
Since $G$ is a $PD^n$ group, we have $H^1(G, \mathcal{F}_H(G))$ $\cong$ $H_{n-1}(G, \ind^G_H (\mathcal{P}H \otimes _{\field} D_G))$. By Shapiro's Lemma, $H_{n-1}(G, \ind^G_H (\mathcal{P}H)\otimes _{\field} D_G))$ $\cong$ $H_{n-1}(H, \mathcal{P}H \otimes_{\field} D_G)$. Since $H$ is a $PD^{n-1}$ group, $H_{n-1}(H, \mathcal{P}H \otimes_{\field} D_G )$ is isomorphic to $\textrm{Hom}_{\field H}(D_H,\mathcal{P}H \otimes_{\field} D_G)$ $\cong \field$. Hence, $\tilde{e}(G,H)= 2$. We now invoke Lemma 2.5 of \cite{KR} to get a subgroup $H'$ of finite index in $H$ such that $e(G,H')$= $\tilde{e}(G,H)$=2. 

Applying Sageev's construction (see \cite{S}) we obtain a CAT(0) cube complex $X$ such that $G$ acts essentially on $X$ and $H'$ is the stabilizer of an oriented codimension 1 hyperplane $J$. As $H'$ has finite index in the property (T) group $H$, $H'$ also has property (T). However, every action of a group with property (T) on a CAT(0) cube complex  must have a fixed point (see \cite{propt}) and so the action of $H'$ on the CAT(0) cube complex $J$ has a global fixed point. Hence, Lemma 2.5 from \cite{fixedpoint} implies the existence of a proper $H'$ almost invariant subset $B$ of $G$ such that $H'BH'=B$. 

Recall that the singularity obstruction $S_B(G,H')$ satisfies the following: for all $g \in S_B(G,H')$, the subgroup $K_g$ defined as $H'\cap gH'g^{-1}$ has a proper almost invariant set $B_g$ such that $K_gB_g=B_g$. But this implies that $e(G,K_g)$ is at least 2. 

Every subgroup of infinite index in an $n$-dimensional Poincar\'e Duality group has cohomological dimension strictly less than $n$ (See \cite{R}). Moreover, for any $PD^n$ group $X$ with subgroup $Y$ of type FP, $\textrm{cd}_{\field} Y$ $\leq n-2$ precisely when $\tilde{e}(X,Y)=1$ (Lemma 5.1 of \cite{torus}). This implies that $K_g$ has finite index in both $H'$ and $gH'g^{-1}$. More precisely, $g$ lies in the commensurator $\textrm{Comm}_G(H')$ of $H'$ and $S_B(G,H)$ is a subset of $\textrm{Comm}_G(H')$. Therefore by Theorem B of \cite{N}, $G$ splits over a subgroup commensurable with $H'$. This proves the theorem. \end{proof}

\end{document}